\documentclass[12pt, reqno]{amsart}
\usepackage[margin=0.75in]{geometry}
\usepackage[
  bookmarks=true]{hyperref}
\usepackage[OT2, T1]{fontenc}
\usepackage[english]{babel}
\usepackage[utf8]{inputenc}
\usepackage{csquotes}
\usepackage[final]{microtype}
\usepackage{lmodern}
\usepackage{amsthm}
\usepackage{amssymb}
\usepackage{mathrsfs}
\usepackage{enumerate}
\usepackage{tikz-cd} 
\usepackage{tikz}
\usepackage{multicol}
\usepackage{multirow}
\usepackage{tikz-qtree}
\usepackage{rotating}
\usepackage{graphicx}
\usepackage{comment}
\usepackage{enumitem}
\usepackage{manfnt}
\usepackage{mathtools}
\usepackage{diagbox}
\usepackage{bbm}
\usetikzlibrary{arrows,calc,matrix,trees,arrows.meta,positioning,decorations.pathreplacing,bending}

\newtheorem{theorem}{Theorem}[section]
\newtheorem*{theorem*}{Theorem}
\newtheorem*{goal*}{Goal}
\newtheorem*{question*}{Question}

\newtheorem{lemma}[theorem]{Lemma}

\theoremstyle{definition}

\newtheorem{remark}[theorem]{Remark}
\newtheorem{definition}[theorem]{Definition}

\newtheorem{condition}[theorem]{Condition}

\newcommand{\genlegendre}[4]{%
  \genfrac{(}{)}{}{#1}{#3}{#4}%
  \if\relax\detokenize{#2}\relax\else_{\!#2}\fi
}

\DeclareSymbolFont{cyrletters}{OT2}{wncyr}{m}{n}
\DeclareMathSymbol{\Sha}{\mathalpha}{cyrletters}{"58}

\usepackage{graphicx}

\title{Large Deviation Principles for Abelian Monoids}
\author{Daniel Keliher}
\address{Daniel Keliher, Massachusetts Institute of Technology, Concourse Program, Cambridge, MA 02139, USA}
\email{keliher@mit.edu}
\urladdr{\url{https://www.danielkeliher.com/}}

\author{Sun Woo Park}
\address{Sun Woo Park, Max Planck Institute for Mathematics,  Vivatsgasse 7, 53111 Bonn, Germany}
\email{s.park@mpim-bonn.mpg.de}
\urladdr{\url{https://sites.google.com/wisc.edu/spark483}}

\date{\today }

\begin{document}

\begin{abstract}
    Following work of Mehrdad and Zhu \cite{MZ} and of Liu \cite{Liu-ErdosKac, Liu-Turan}, we prove a large deviation principle for a broad class of integer-valued additive functions defined over abelian monoids. As a corollary, we obtain a large deviation principle for a generalized form of the Erd{\H{o}s}-Kac theorem due to Liu. 
\end{abstract}

\maketitle

\section{Introduction}

The celebrated Erd\H{o}s-Kac Theorem  \cite{EK40} says that if $\omega(n)$ is the number of distinct prime factors of a positive integer $n$, then 
\begin{equation} \label{eq:EK}
    \frac{\omega(n)-\log \log n}{\sqrt{\log \log n}}
\end{equation}
is normally distributed with mean $0$ and standard deviation 1 (cf. Theorem \ref{thm:LiuEK}).

Much work has been done on the rate at which \eqref{eq:EK} converges to $N(0,1)$, see  \cite{RT58} for one such example. Likewise, one can study the tails of the distribution as the upper bound for $n$ grows.  Indeed, in \cite{MZ}, Mehrdad and Zhu prove a large deviation principle for a large class of ``strongly additive'' functions $g : \mathbb{N} \to \mathbb{N}$, including when $g$ is the number-of-distinct-prime-divisors function, $\omega$, thus giving a large deviation principle for the distributions of \eqref{eq:EK} over $n \leq X$ as $X \to \infty$. 

In a different direction, Liu proves a generalization of the Erd{\H{o}}s-Kac theorem for classes of abelian monoids $\mathcal{M}$ generated by a set $\mathcal{P}$ \cite{Liu-ErdosKac}. This general framework includes, for example, the cases where $\mathcal{P}$ is the set of rational primes, the set of irreducible monic polynomials over a finite field, the set of effective 0-cycles of a nice variety over a finite field,  and  the case where $\mathcal{P}$ is a Beurling system of generalized primes. See  \cite{B37} or \cite{BeurlingBook} for background on the latter. 

We will work in this same general setting, which we now describe.  Let $\mathcal{P}$ be a set together with a ``norm'' map ${\mathbf{N}: \mathcal{P} \to \mathbb{N}\setminus\{1\}.}$ Let $\mathcal{M}$ be the multiplicative monoid generated by the elements of $\mathcal{P}$. That is, every element $m \in \mathcal{M}$ can be expressed as $m = \prod_{p \in \mathcal{P}}p^{e_p}$ where almost all of the $e_p \in \mathbb{Z}_{\geq 0}$ are zero.   Extend $\mathbf{N}$ to a monoid morphism $\mathcal{M} \to \mathbb{N}$ by taking, for any $m \in \mathcal{M}$,  ${\mathbf{N}(m) = \prod_{p \in \mathcal{P}}\mathbf{N}(p)^{e_P}.}$

Throughout, we make the following two assumptions on the growth rate of $\mathcal{M}$ and $\mathcal{P}$ with respect to $\mathbf{N}$. 

\begin{condition}  \label{condition:Beurling}
    For $\mathcal{P}, \mathcal{M}, \mathbf{N}$ as above, assume
    \begin{enumerate}
    \item ${\displaystyle \sum_{\substack{ m \in \mathcal{M} \\ \mathbf{N}(m) \leq X}} 1 = a X + O(X^b)}$ for some $a > 0$ and $b \in [0,1)$, and  
    \item  ${\displaystyle \sum_{\substack{ p \in \mathcal{P} \\ \mathbf{N}(p) \leq X}} 1 = O\left( \frac{X}{\log X}\right)}$.
\end{enumerate}
\end{condition}

In this setting, Liu proves the following generalization of the Erd{\H{o}}s-Kac Theorem. 
 
\begin{theorem}[Theorem 1, \cite{Liu-ErdosKac}]\label{thm:LiuEK}
    If $\mathcal{M}, \mathcal{P}, \mathbf{N}$ are as above and satisfy Condition \ref{condition:Beurling}, then
    \begin{equation}\label{eq:LiuEK}
        \lim_{X \to \infty} \frac{\#\left\{m : \mathbf{N}(m) \leq X, \left| \frac{\omega(m) - \log \log \mathbf{N}(m)}{\sqrt{\log \log \mathbf{N}(m)}} \right| \leq \gamma \right\}}{\#\{m : \mathbf{N}(m) \leq X\}} = \int_{-\infty}^\gamma  e^{-t^2/2}dt.
    \end{equation}
\end{theorem}

Our goal is to prove a generalization of the large deviation principle in \cite{MZ} to the setting of abelian monoids considered by Liu. In particular, the large deviation principle in our main theorem, Theorem \ref{thm:Main-MonoidMZ} below, applies to the distributions \eqref{eq:LiuEK}. 

Our large deviation result will apply to a class of ``strongly additive'' functions. 

\begin{condition} \label{condition:additive}
Let $g: \mathcal{M} \to \mathbb{N}$ be a strongly additive function, i.e. $g$ satisfies the following two conditions:
\begin{align*}
    g(p^k) &= g(p) \text{ for all prime elements } p \in \mathcal{P}, \\
    g(mn) &= g(m) + g(n) \text{ for all } m,n \in \mathcal{M} \text{ such that } (m,n) = 1.
\end{align*}
We suppose further that there exists a probability measure $\rho$ on $\mathbb{R}$ satisfying the following two conditions.
\begin{enumerate}
    \item For any $\theta \in \mathbb{R}$, we have $\int_\mathbb{R} e^{\theta y} \rho(dy) < \infty$.
    \item Let $A \subset \mathbb{R}$ be any Borel measurable set. Define probability measures $\rho_X$ given by
    \begin{equation}
        \rho_X(A) := \frac{\displaystyle \sum_{\substack{g(p) \in A \\ \mathbf{N}(p) \leq X}} \frac{1}{\mathbf{N}(p)}}{\displaystyle\sum_{\mathbf{N}(p) \leq X} \frac{1}{\mathbf{N}(p)}}.
    \end{equation}
    Then for any $\theta \in \mathbb{R}$, we have $\int_\mathbb{R} e^{\theta y} \rho_X(dy) \to \int_\mathbb{R} e^{\theta y} \rho(dy)$.
\end{enumerate}
\end{condition}

The application we have in mind is when $g$ is the number-of-distinct-prime-divisors function, $\omega$. On $\mathcal{M}$, this means for $m = \prod_{p \in \mathcal P}p^{e_p}$, let $\omega(m)$ be the number of $p\in \mathcal{P}$ for which $e_p \geq 1$.  
 
\begin{definition} \label{def:VZ}
    For each $n  \in \mathbb{N}$, we denote by $V(n)$ a uniformly chosen monoid element from the set $\{m \in \mathcal{M} : \mathbf{N}(m) \leq n\}$. For each $p \in \mathcal{P}$, let $Z$ be a random variable such that $Z_p = 1$ if $V(n)$ is divisible by $p$, and $Z_p = 0$ otherwise. 
\end{definition}

We are now ready to state the main theorem. 

\begin{theorem}\label{thm:Main-MonoidMZ}
    Let $\mathcal{M}$ be a multiplicative abelian monoid generated by a set $\mathcal{P}$ and endowed with a norm function $\mathbf{N}: \mathcal{M} \to \mathbb{N}$, all satisfying Condition \ref{condition:Beurling}. Let $g: \mathcal{M} \to \mathbb{N}$ be an arithmetic function satisfying Condition \ref{condition:additive}. Let $W$ be a random variable over $\mathbb{N}$ defined as
    \begin{equation}
        W(n) := g(V(n)),
    \end{equation}
    where $V(n)$ is a uniformly chosen monoidal element from the set $\{m \in \mathcal{M}: \mathbf{N}(m) \leq n\}$, see Definition \ref{def:VZ}.
    Then for any Borel measurable set $A \subset \mathbb{R}$, the probability $\mathbb{P}\left[ \frac{W(X)}{ \log \log X} \in A \right]$ satisfies a large deviation principle with speed $\log \log X$ and rate function
    \begin{equation*}
        I(x) := \sup_{\theta \in \mathbb{R}} \left\{ \theta x - \int_\mathbb{R} (e^{\theta y} - 1) \rho (dy) \right\}.
    \end{equation*}
    More explicitly, we have for any Borel measurable set $A \subset \mathbb{R}$,
    \begin{align*}
        -\inf_{x \in A^o} I(x) &\leq \liminf_{X \to \infty} \frac{1}{\log \log X} \log \mathbb{P}\left[ \frac{W(X)}{ \log \log X} \in A \right] \\
        &\leq \limsup_{X \to \infty} \frac{1}{\log \log X} \log \mathbb{P}\left[ \frac{W(X)}{ \log \log X} \in A \right] \leq -\inf_{x \in \overline{A}} I(x),
    \end{align*}
    where $A^o$ is the interior of $A$, and $\overline{A}$ is the closure of $A$.
\end{theorem}

The strategy of the proof will closely follow that employed in \cite{MZ}, but with suitable modifications that accommodate the more general setting.

\section{Applications}\label{sec:applications}
Note that taking $\mathcal{P}$ to be the set of rational primes and $\mathbf{N}(p) = p$, Theorem \ref{thm:Main-MonoidMZ} recovers exactly Theorem 2 of Mehrdad and Zhu \cite{MZ}. Furthermore, taking $\mathcal{P}$ to be the set of monic, irreducible polynomials $f$ over a finite field $\mathbb{F}_q$ and $\mathbf{N}(f) = q^{\deg f}$ recovers the large deviation result of Feng, Wang, and Yang \cite{FWY20}. 

Two additional settings are natural to consider for applying Theorem \ref{thm:Main-MonoidMZ} in conjunction with Theorem \ref{thm:LiuEK}:

\begin{itemize}
    \item \textit{Number Fields:} Let $\mathcal{O}_K$ be the ring of integers of a number field, $K$. Let $\mathcal{P}$ be the set of prime ideals of $\mathcal{O}_K$, so $\mathcal{M}$ is the set of ideals of $\mathcal{O}_K$. As usual, for $\mathfrak{p} \in \mathcal P$, let $\mathbf{N}(\mathfrak{p})=|\mathcal{O}_K/\mathfrak{p}|$. Condition \ref{condition:additive} is satisfied by work of Weber \cite{Weber96}, see also \cite[Chapter XIII]{LangANT94} (for I), and by the prime ideal theorem (for II). 

\item \textit{Nice Varieties:}
Let $V$ be any $d$-dimensional smooth, projective irreducible, geometrically integral variety defined over a finite field, $\mathbb{F}_q$. Let $\mathcal{P}$ be the set of closed points of $V/\mathbb{F}_q$, so $\mathcal{M}$ is the monoid of effective $0$-cycles. For $P \in \mathcal{P}$, take $\mathbf{N}(P) = q^{d\cdot\deg(P)}$, where $\deg(P)$ is the size of the Galois orbit of $P$ in $V(\mathbb{\bar F}_q)$.  Condition \ref{condition:additive} holds for this choice of $\mathcal{P},\mathcal{M},\mathbf{N}$ by \cite{LangWeil54}; see \cite{Rosen02} for an exposition of the case of curves over $\mathbb{F}_q$.
\end{itemize} 

In the examples above, taking $g$ to be the appropriate number-of-distinct-prime-divisors function, $\omega$, on $\mathcal{M}$, Theorem \ref{thm:Main-MonoidMZ} gives a large deviation principle with an explicit rate function for the corresponding version of the Erd{\H{o}}s-Kac Theorem given by Theorem \ref{thm:LiuEK}. We can use \cite[Corollary 3]{MZ} to explicitly compute the rate function $I(x)$ as
\begin{equation}
    I(x) := \begin{cases}
        x \log x - x + 1 &\text{ if } x \geq 0, \\
        +\infty &\text{ otherwise}.
    \end{cases}
\end{equation}

\section{Preliminary Results}\label{sec:lemmas}

In this section we gather some results and prove some lemmas in preparation for the proof of Theorem \ref{thm:Main-MonoidMZ}.

The following version of Mertens' second theorem will be a useful estimate. 
\begin{lemma}[Lemma 2, \cite{Liu-Turan}] \label{lemma:Liu-Turan}
    If $\mathcal{P},\mathcal{M}, \mathbf{N}$ satisfy Condition \ref{condition:Beurling}, then 
    $$\sum_{ \substack{ p \in \mathcal{P} \\ \mathbf{N}(p) < X}} \frac{1}{\mathbf{N}(p)} = \log \log X + O\left(\frac{1}{\log X}\right).$$
\end{lemma}

Notice Lemma \ref{lemma:Liu-Turan} implies that the number-of-distinct-prime-divisors function $\omega$ on $\mathcal{M}$ satisfies Condition \ref{condition:additive}.

We next state a version of G\"artner-Ellis Theorem as given in \cite{MZ}; it is this result from which we will obtain the desired large deviation principle. 

\begin{theorem}[Page 152 of \cite{MZ}, and Theorem 2.3.6 of \cite{DZ}] \label{theorem:Gartner-Ellis}
    Let $Z_n$ be a sequence of random variables on $\mathbb{R}$. Let $a_n$ be a sequence of positive numbers such that $\lim_{n \to \infty} a_n = \infty$. Suppose that for any $\theta \in \mathbb{R}$, the limit
    \begin{equation*}
        \Lambda(\theta) := \lim_{n \to \infty} \frac{1}{a_n} \log \mathbb{E}[\exp \left( \theta a_n Z_n \right)]
    \end{equation*}
    exists and is differentiable for every $\theta \in \mathbb{R}$. Then for any Borel measurable set $A \subset \mathbb{R}$, the probability $\mathbb{P}[Z_n \in A]$ satisfies a large deviation principle with speed $a_n$ and rate function
    \begin{equation*}
        I(x) := \sup_{\theta \in \mathbb{R}} \left\{ \theta x - \Lambda(\theta) \right\}.
    \end{equation*}
    More concretely, we have for any Borel measurable set $A \subset \mathbb{R}$,
    \begin{align*}
        -\inf_{x \in A^o} I(x) \leq \liminf_{n \to \infty} \frac{1}{a_n} \log \mathbb{P}[Z_n \in A] \leq \liminf_{n \to \infty} \frac{1}{a_n} \log \mathbb{P}[Z_n \in A] \leq -\inf_{x \in \overline{A}} I(x)
    \end{align*}
\end{theorem}

For every element $p \in \mathcal{P}$, define independent random variables $Y_p$ for which 
$$Y_p = \begin{cases}
    1 \quad&\text{with probability }\frac{1}{\mathbf{N}(p)} \\
    0 \quad&\text{with probability } 1- \frac{1}{\mathbf{N}(p)}
\end{cases}.$$

\begin{lemma} \label{lemma:evZpYp}
Given $p \in \mathcal{P}$, we recall the random variable $Z_p$ supported over the set $\{m \in \mathcal{M} : \mathbf{N}(m) \leq X\}$ from Definition \ref{def:VZ}.
For distinct $p_1,...,p_k \in \mathcal{P}$ we have 
    \begin{equation}\label{eq:evZp}
        \mathbb{E}[Z_{p_1}Z_{p_2}...Z_{p_k}] = \frac{\#\{m \in \mathcal{M} : \mathbf{N}(m) \leq \lfloor \frac{X}{\mathbf{N}(p_1) \mathbf{N}(p_2) \cdots \mathbf{N}(p_k)} \rfloor \}}{\#\{ m \in \mathcal{M} : \mathbf{N}(m) \leq X \}},
    \end{equation}
    and 
    \begin{equation}\label{eq:evYp}
        \mathbb{E}[Y_{p_1}Y_{p_2}...Y_{p_k}] = \frac{1}{\mathbf{N}(p_1) \mathbf{N}(p_2) \cdots \mathbf{N}(p_k)}.
    \end{equation}
\end{lemma}
\begin{proof}
    Given $m \in \mathcal{M}$, we have $Z_{p_1}(m) Z_{p_2}(m) \cdots Z_{p_k}(m) = 1$ if and only if there exists an element $a \in \mathcal{M}$ such that $m = p_1 p_2 \cdots p_k a$. For such an $m$, we have $\mathbf{N}(m) \leq X$ if and only if $\mathbf{N}(a) \leq \lfloor \frac{X}{\mathbf{N}(p_1) \mathbf{N}(p_2) \cdots \mathbf{N}(p_k)} \rfloor$. Hence, we have
    \begin{align*}
        \mathbb{E}[Z_{p_1} Z_{p_2} \cdots Z_{p_k}] &= \frac{\sum_{\substack{m \in \mathcal{M} \\ \mathbf{N}(m) \leq X}} \mathbbm{1}_{p_1p_2 \cdots p_k \mid m} }{\# \{m \in \mathcal{M} : \mathbf{N}(m) \leq X\}} \\
        &= \frac{\#\{m \in \mathcal{M}: m = p_1 p_2 \cdots p_k a, \mathbf{N}(m) \leq X\}}{\# \{m \in \mathcal{M} : \mathbf{N}(m) \leq X\}} \\
        &= \frac{\#\{m \in \mathcal{M} : \mathbf{N}(m) \leq \lfloor \frac{X}{\mathbf{N}(p_1) \mathbf{N}(p_2) \cdots \mathbf{N}(p_k)} \rfloor \}}{\#\{ m \in \mathcal{M} : \mathbf{N}(m) \leq X \}}.
    \end{align*}
    The second equation follows immediately from the definition of $Y_p$.
\end{proof}

\begin{lemma} \label{lemma:evZpYp-condition}
    Suppose Condition \ref{condition:Beurling}.
    \begin{enumerate}
    \item For distinct primes $p_1, \cdots, p_k \in \mathcal{P}$, we have
        \begin{equation*}
        \mathbb{E}[Z_{p_1}Z_{p_2}...Z_{p_k}] = \frac{1}{aX} \lfloor \frac{aX}{\mathbf{N}(p_1) \mathbf{N}(p_2) \cdots \mathbf{N}(p_k)} \rfloor + O \left( \frac{X^{b-1}}{\mathbf{N}(p_1)^b \mathbf{N}(p_2)^b \cdots \mathbf{N}(p_k)^b} \right).
    \end{equation*}
    \item For distinct primes $p_1, \cdots, p_k \in \mathcal{P}$, we have $$\mathbb{E}[Z_{p_1}Z_{p_2}...Z_{p_k}] \leq \mathbb{E}[Y_{p_1}Y_{p_2}...Y_{p_k}] + O \left( \frac{X^{b-1}}{\mathbf{N}(p_1)^b \mathbf{N}(p_2)^b \cdots \mathbf{N}(p_k)^b} \right).$$
    \item For sufficiently large $X$ there exists an explicit positive constant $M$, independent of the choice of distinct primes $p_1, \cdots, p_k \in \mathcal{P}$, such that $$\mathbb{E}[Z_{p_1}Z_{p_2}...Z_{p_k}] \leq M \cdot \mathbb{E}[Y_{p_1}Y_{p_2}...Y_{p_k}].$$
    \item For sufficiently large $X$, there exists an explicit constant $M$ such that for any non-negative sequence of real numbers $\{\theta_p\}_{p \in \mathcal{P}}$ we have
    $$\mathbb{E} \left[\exp \left( \sum_{\mathbf{N}(p) \leq X} \theta_p Z_p \right) \right] \leq M \cdot \mathbb{E} \left[\exp \left( \sum_{\mathbf{N}(p) \leq X} \theta_p Y_p \right)\right].$$
    \end{enumerate}
    The implied constants of all the error terms are independent of the choice of $p_1, \cdots, p_k$.
\end{lemma}
\begin{proof}
    By Condition \ref{condition:Beurling} and Lemma \ref{lemma:evZpYp}, we have
    \begin{align*}
        \mathbb{E}[Z_{p_1} Z_{p_2} \cdots Z_{p_k}] &= \frac{\lfloor \frac{aX}{\mathbf{N}(p_1) \mathbf{N}(p_2) \cdots \mathbf{N}(p_k)} \rfloor + O \left( \frac{X^b}{\mathbf{N}(p_1)^b \mathbf{N}(p_2)^b \cdots \mathbf{N}(p_k)^b} \right)}{aX + O(X^b)} \\
        &= \frac{1}{aX} \left\lfloor \frac{aX}{\mathbf{N}(p_1) \mathbf{N}(p_2) \cdots \mathbf{N}(p_k)} \right\rfloor + O \left( \frac{X^{b-1}}{\mathbf{N}(p_1)^b \mathbf{N}(p_2)^b \cdots \mathbf{N}(p_k)^b} \right).
    \end{align*}
    The second statement follows from using the inequality
    \begin{equation*}
        \frac{1}{aX} \left\lfloor \frac{aX}{\mathbf{N}(p_1) \mathbf{N}(p_2) \cdots \mathbf{N}(p_k)} \right\rfloor \leq \frac{1}{\mathbf{N}(p_1) \mathbf{N}(p_2) \cdots \mathbf{N}(p_k)} = \mathbb{E}[Y_{p_1} Y_{p_2} \cdots Y_{p_k}].
    \end{equation*}
    To prove the third statement, we divide into two cases. Suppose $p_1, p_2, \cdots, p_k$ are distinct primes such that $\mathbf{N}(p_1) \mathbf{N}(p_2) \cdots \mathbf{N}(p_k) > X$. Then because all the random variables $Z_{p_1}, Z_{p_2}, \cdots, Z_{p_k}$ are supported over the probability space $\{m \in \mathcal{M} : \mathbf{N}(m) \leq X\}$, we have $$\mathbb{E}[Z_{p_1} Z_{p_2} \cdots Z_{p_k}] = 0 < \frac{1}{\mathbf{N}(p_1) \mathbf{N}(p_2) \cdots \mathbf{N}(p_k)} = \mathbb{E}[Y_{p_1} Y_{p_2} \cdots Y_{p_k}].$$ Now suppose $p_1, p_2, \cdots, p_k$ are distinct primes such that $\mathbf{N}(p_1) \mathbf{N}(p_2) \cdots \mathbf{N}(p_k) \leq X$. Then we have
    \begin{align*}
        \frac{X^{b-1}}{\mathbf{N}(p_1)^b \mathbf{N}(p_2)^b \cdots \mathbf{N}(p_k)^b} &= \frac{1}{X^{1-b}\mathbf{N}(p_1)^b \mathbf{N}(p_2)^b \cdots \mathbf{N}(p_k)^b} \\
        &\leq \frac{1}{\mathbf{N}(p_1) \mathbf{N}(p_2) \cdots \mathbf{N}(p_k)} = \mathbb{E}[Y_{p_1} Y_{p_2} \cdots Y_{p_k}].
    \end{align*}
    Hence, by the second statement of our lemma, there exists an absolute constant $M_1 > 0$ such that $$\mathbb{E}[Z_{p_1} Z_{p_2} \cdots Z_{p_k}] \leq M_1 \mathbb{E}[Y_{p_1} Y_{p_2} \cdots Y_{p_k}].$$ We take $M := \max(1, M_1)$ to obtain the third statement of the lemma.
    
    The fourth statement follows by using Taylor expansion. We crucially use the fact that for any non-negative integers $r_1, r_2 \cdots, r_k$ we have
    \begin{equation*}
        \mathbb{E}[Z_{p_1}^{r_1} Z_{p_2}^{r_2} \cdots Z_{p_k}^{r_k}] = \mathbb{E}[Z_{p_1} Z_{p_2} \cdots Z_{p_k}] \leq M \cdot \mathbb{E}[Y_{p_1} Y_{p_2} \cdots Y_{p_k}] = M \cdot \mathbb{E}[Y_{p_1}^{r_1} Y_{p_2}^{r_2} \cdots Y_{p_k}^{r_k}].
    \end{equation*}
\end{proof}

The next three lemmas generalize Lemmas 7, 8, and 9 of \cite{MZ} for the abelian monoidal setting. 

\begin{lemma}\label{lemma:MZ7} For any $\varepsilon > 0$, 
$$\limsup_{C \to \infty} \limsup_{X \to \infty} \frac{1}{\log \log X} \log \mathbb{P}\Bigg( \bigg| \sum_{\substack{ p \in \mathcal{P} \\ g(p) > C \\
\mathbf{N}p < X}} g(p)Z_p  \bigg| \geq \varepsilon \log \log X\Bigg) = -\infty.$$
\end{lemma}

\begin{proof}
    If
    \begin{equation}\label{eq:logprobtail1}
        \limsup_{C \to \infty} \limsup_{X \to \infty} \frac{1}{\log \log X} \log \mathbb{P}\Bigg(  \sum_{\substack{ p \in \mathcal{P} \\ g(p) > C \\
        \mathbf{N}p < X}} g(p)Z_p   \geq \varepsilon \log \log X\Bigg) = -\infty,
    \end{equation}
    and 
    \begin{equation}\label{eq:logprobtail2}
        \limsup_{C \to \infty} \limsup_{X \to \infty} \frac{1}{\log \log X} \log \mathbb{P}\Bigg(  \sum_{\substack{ p \in \mathcal{P} \\ g(p) > C \\
        \mathbf{N}p < X}} g(p)Z_p   \leq  -\varepsilon \log \log X\Bigg) = -\infty,
    \end{equation}
    then the conclusion of the lemma follows. We begin by proving \eqref{eq:logprobtail1}. 

    By the exponential\footnote{For an integrable random variable  $X$ with finite, non-zero variance  and $\varepsilon > 0$,  $P(X\geq \varepsilon )\leq e^{-\theta \varepsilon} \mathbb{E}(e^{\theta X})$ for any $\theta>0.$} Chebyshev's Inequality with any $\theta > 0$,  and the third statement of Lemma \ref{lemma:evZpYp-condition}, there exists an explicit constant $M > 0$ such that 
    \begin{equation}\label{eq:expcheby}
    \begin{split}
        & \hspace{14pt} \limsup_{X \to \infty} \frac{1}{\log \log X} \log \mathbb{P}\left( \sum_{\substack{g(p)>C \\ \mathrm{N(p) < X}}} g(p) Z_p \geq \varepsilon\log \log X\right) \\ 
        &\leq 
         \limsup_{X \to \infty} \frac{1}{\log \log X} \log \mathbb{E}\left[  \exp\left(\sum_{\substack{g(p)>C \\ \mathrm{N(p) < X}}} g(p) Z_p\right)  \right]- \theta \varepsilon \\
         &\leq \limsup_{X \to \infty} \frac{1}{\log \log X} \left(\log \mathbb{E}\left[  \exp\left(\sum_{\substack{g(p)>C \\ \mathrm{N(p) < X}}} g(p) Y_p\right)\right]  + \log M \right)- \theta \varepsilon \\
         &= 
         \limsup_{X \to \infty} \log \mathbb{E} \frac{1}{\log \log X}\left[  \exp\left(\sum_{\substack{g(p)>C \\ \mathrm{N(p) < X}}} g(p) Y_p\right)  \right] - \theta \varepsilon.
     \end{split}
     \end{equation}

Applying Lemma \ref{lemma:Liu-Turan}, $\sum_{\mathbf{N}p \leq X} \frac{1}{\mathbf{N}(p)} \approx \log \log X$, and then using that $\log (x+1) \leq x$, we have

\begin{equation}\label{eq:applymertens1}
\begin{split}
& \hspace{14pt} \limsup_{X \to \infty} \log \mathbb{E} \frac{1}{\log \log X} \left[  \exp\left(\sum_{\substack{g(p)>C \\ \mathrm{N(p) < X}}} g(p) Y_p\right)  \right]- \theta \varepsilon  \\ &\leq \limsup_{X \to \infty}\frac{\sum_{\substack{g(p) > C \\ \mathbf{N}(p) \leq X } } \log (e^{\theta g(p)}-1)\mathbf{N}(p)^{-1}+1 )}{\sum_{\mathbf{N}(p) \leq X}  \mathbf{N}(p)^{-1}} - \theta \varepsilon \\
&\leq \limsup_{X \to \infty}\frac{\sum_{\substack{g(p) > C \\ \mathbf{N}(p) \leq X } } (e^{\theta g(p)}-1)\mathbf{N}(p)^{-1}}{\sum_{\mathbf{N}(p) \leq X}  \mathbf{N}(p)^{-1}} - \theta \varepsilon.
\end{split}
\end{equation}

Finally, we apply Condition \ref{condition:additive} (2) to the last term to obtain
\begin{equation}\label{eq:cond2int}
    \limsup_{X \to \infty}\frac{\sum_{\substack{g(p) > C \\ \mathbf{N}(p) \leq X } } (e^{\theta g(p)}-1)\mathbf{N}(p)^{-1}}{\sum_{\mathbf{N}(p) \leq X}  \mathbf{N}(p)^{-1}} \leq \limsup_{X \to \infty} \int_{y > C} (e^{\theta y}-1)\rho_X(dy) \leq \int_{y > C}e^{\theta y -1} \rho(dy).
\end{equation}
Notice $\int_{y>C}e^{\theta y -1} \rho(dy) \to 0$ as $C \to \infty$. Taken together, \eqref{eq:expcheby}, \eqref{eq:applymertens1}, and \eqref{eq:cond2int} show

       $$ \limsup_{C \to \infty} \limsup_{X \to \infty} \frac{1}{\log \log X} \log \mathbb{P}\Bigg(  \sum_{\substack{ p \in \mathcal{P} \\ g(p) > C \\
        \mathbf{N}p < X}} g(p)Z_p   \geq \varepsilon \log \log X\Bigg) = -\theta \varepsilon. $$
Since this holds for all $\theta > 0$, \eqref{eq:logprobtail1} follows. The proof of \eqref{eq:logprobtail2} is nearly identical. 
\end{proof}

Now, define $k_X := X^{\frac{1}{(\log \log X)^2}}$.

\begin{lemma}\label{lemma:MZ8}
    Let $A(X,C) := \{p \in \mathcal{P} : k_X \leq \mathbf{N}(p) \leq X, |g(p)| \leq C\}$.    For any $\epsilon > 0$, we have
    \begin{equation}
        \limsup_{X \to \infty} \frac{1}{\log \log X} \log \mathbb{P} \left[ \left| \sum_{p \in A(X,C)} g(p)Z_p \right| \geq \epsilon \log \log X \right] = -\infty.
    \end{equation}
\end{lemma}
\begin{proof}
    By Condition \ref{condition:Beurling} and the third statement of Lemma \ref{lemma:evZpYp-condition}, for any $\theta > 0$ and sufficiently large $X$ there exists an explicit constant $M > 0$ such that
    \begin{align*}
        \log \mathbb{E} \left[ \exp \left( \theta \left| \sum_{p \in A(X,C)} g(p)Z_p \right| \right) \right]
        &\leq \log \mathbb{E} \left[ \exp \left( \theta C \sum_{p \in A(X,C)} Y_p \right) \right] + \log M \\
        &= \sum_{p \in A(X,C)} \log \left( (\exp \left(\theta C \right) - 1) \cdot \frac{1}{\mathbf{N}(p)} + 1 \right) + \log M \\
        &\leq (\exp(\theta C) - 1) \cdot \sum_{k_X \leq \mathbf{N}(p) \leq X} \frac{1}{\mathbf{N}(p)} + \log M.
    \end{align*}
    By Lemma \ref{lemma:Liu-Turan}, we have
    \begin{equation*}
        \sum_{k_X \leq \mathbf{N}(p) \leq X} \frac{1}{\mathbf{N}(p)} = 2 \log \log \log X + O \left( \frac{(\log \log X)^2}{\log X} \right).
    \end{equation*}
    We then use exponential Chebyshev's inequality to conclude

    \begin{align*}
        & \hspace{14pt} \limsup_{X \to \infty} \frac{1}{\log \log X} \log \mathbb{P} \left[ \left| \sum_{p \in A(X,C)} g(p)Z_p \right| \geq \epsilon \log \log X \right] \\
        &\leq \limsup_{X \to \infty} \frac{1}{\log \log X} \log \mathbb{E} \left[ \exp \left( \theta \left| \sum_{p \in A(X,C)} g(p)Z_p \right| \right) \right] - \theta \epsilon \\
        &\leq -\theta \epsilon.
    \end{align*}
    Let $\theta$ grow arbitrarily large to conclude the lemma.
\end{proof}

\begin{lemma}\label{lemma:MZ9}
    Denote by $B(X,C) := \{p \in \mathcal{P} : \mathbf{N}(p) \leq k_X, |g(p)| \leq C \}$. Then for any $\theta \in \mathbb{R}$, we have
    \begin{equation}
        \lim_{X \to \infty} \left| \mathbb{E} \left[ \exp \left( \theta \sum_{p \in B(X,C)} g(p) Z_p \right) \right] - \mathbb{E} \left[ \exp \left( \theta \sum_{p \in B(X,C)} g(p) Y_p \right) \right] \right| = 0.
    \end{equation}
\end{lemma}
\begin{proof}
    Let $K$ be any constant. Consider the following three expressions:
    \begin{equation} \label{eq:lemma-1}
        \sum_{r \leq K \log \log X} \frac{|\theta|^r}{r!} \left| \mathbb{E} \left[ \left( \sum_{p \in B(X,C)} g(p) Z_p \right)^r \right] - \mathbb{E} \left[ \left( \sum_{p \in B(X,C)} g(p) Y_p \right)^r \right] \right|,
    \end{equation}
    \begin{equation} \label{eq:lemma-2}
        \sum_{r > K \log \log X} \frac{|\theta|^r}{r!} \mathbb{E} \left[ \left( \sum_{p \in B(X,C)} g(p) Z_p \right)^r \right],
    \end{equation}
    \begin{equation} \label{eq:lemma-3}
        \sum_{r > K \log \log X} \frac{|\theta|^r}{r!} \mathbb{E} \left[ \left( \sum_{p \in B(X,C)} g(p) Y_p \right)^r \right].
    \end{equation}
    Then by Taylor expansion we have
    \begin{align*}
        & \hspace{14pt} \left| \mathbb{E} \left[ \exp \left( \theta \sum_{p \in B(X,C)} g(p) Z_p \right) \right] - \mathbb{E} \left[ \exp \left( \theta \sum_{p \in B(X,C)} g(p) Y_p \right) \right] \right| \leq \eqref{eq:lemma-1} + \eqref{eq:lemma-2} + \eqref{eq:lemma-3}.
    \end{align*}
    By the second statement of Lemma \ref{lemma:evZpYp-condition}, for any tuples of non-negative integers $r_1, \cdots, r_k$ and sufficiently large $X$ there exists an absolute constant $B > 0$ such that
    \begin{align*}
        \left| \mathbb{E} \left[ Z_{p_1}^{r_1} Z_{p_2}^{r_2} \cdots Z_{p_k}^{r_k} \right] - \mathbb{E} \left[ Y_{p_1}^{r_1} Y_{p_2}^{r_2} \cdots Y_{p_k}^{r_k} \right] \right| = \left| \mathbb{E} \left[ Z_{p_1} Z_{p_2} \cdots Z_{p_k} \right] - \mathbb{E} \left[ Y_{p_1} Y_{p_2} \cdots Y_{p_k} \right] \right| \leq B X^{b-1}.
    \end{align*}
    Therefore, we have
    \begin{align*}
        & \hspace{14pt} \left| \mathbb{E} \left[ \left( \sum_{p \in B(X,C)} g(p) Z_p \right)^r \right] - \mathbb{E} \left[ \left( \sum_{p \in B(X,C)} g(p) Y_p \right)^r \right] \right| \\
        &\leq \left| \sum_{k=1}^r \sum_{r_i} \frac{r!}{r_1! r_2! \cdots r_k!} \frac{1}{k!} \sum_{p_j} g(p_1)^{r_1} \cdots g(p_k)^{r_k} \cdot \left(\mathbb{E} \left[ Z_{p_1}^{r_1} Z_{p_2}^{r_2} \cdots Z_{p_k}^{r_k} \right] - \mathbb{E} \left[ Y_{p_1}^{r_1} Y_{p_2}^{r_2} \cdots Y_{p_k}^{r_k} \right] \right) \right| \\
        &\leq \sum_{k=1}^r \sum_{r_i} \frac{r!}{r_1! r_2! \cdots r_k!} \frac{1}{k!} \sum_{p_j} \frac{B \cdot C^{r_1 + \cdots + r_k}}{X^{1-b}} \leq \frac{B}{X^{1-b}} \left( \sum_{p \in B(X,C)} C \right)^r \leq \frac{B \cdot (C \cdot k_X)^r}{X^{1-b}},
    \end{align*}
    where the $r_i$'s range over tuples of integers $(r_1, \cdots, r_k)$ such that $r_1 + \cdots + r_k = r$, and the $p_j$'s ranges over prime elements $p_1, \cdots, p_k \in B(X,C)$. By Stirling's approximation for $r!$, there exists an absolute constant $B' > 0$ such that
    \begin{align*}
        \eqref{eq:lemma-1} &\leq \sum_{r \leq K \log \log X} \frac{|\theta|^r}{r!} \frac{B \cdot (C \cdot k_X)^r}{X^{1-b}} \\
        &\leq B' \cdot \exp \left(- (1-b) \log X \right) \cdot K \log \log X \cdot \left[\max_{r \leq K \log \log X} \left( \frac{e |\theta| C k_X}{r} \right)^r \right] \\
        &\leq B' \cdot \exp \left(- (1-b) \log X \right) \cdot K \log \log X \cdot \left( \frac{e |\theta| C k_X}{K \log \log X} \right)^{K \log \log X} \\
        &\leq B' K \cdot \exp \left(- (1-b) \log X + \log \log \log X + K (\log |\theta| + \log C) \log \log X + \frac{\log X}{\log \log X} \right).
    \end{align*}
    Hence, there exists an explicit constant $B'' > 0$ such that for sufficiently large $X$,
    \begin{equation*}
        \eqref{eq:lemma-1} \leq B'' e^{-\frac{1}{2}(1-b) \log X}.
    \end{equation*}
    Note that the third statement of Lemma \ref{lemma:evZpYp-condition} implies that for sufficiently large $X$ there exists an explicit constant $M > 0$ such that
    \begin{equation*}
        \eqref{eq:lemma-2} \leq M \cdot \eqref{eq:lemma-3},
    \end{equation*}
    so it suffices to consider \eqref{eq:lemma-3} for sufficiently large $X$. We have
    \begin{align*}
        \eqref{eq:lemma-3} &\leq \sum_{r > K \log \log X} \frac{|\theta|^r}{r!} \cdot C^r \cdot \left( \left( \sum_{\mathbf{N}p \leq k_X} \frac{1}{\mathbf{N}p} \right) + \left( \sum_{\mathbf{N}p \leq k_X} \frac{1}{\mathbf{N}p} \right)^2 + \cdots + \left( \sum_{\mathbf{N}p \leq k_X} \frac{1}{\mathbf{N}p} \right)^r \right) \\
        &\leq \sum_{r > K \log \log X} \frac{|\theta|^r}{r!} \cdot r \cdot C^r \cdot \left( \sum_{\mathbf{N}p \leq k_X} \frac{1}{\mathbf{N}p} \right)^r.
    \end{align*}
    By Lemma \ref{lemma:Liu-Turan} and Stirling's approximation for $r!$, we have
    \begin{align*}
        \eqref{eq:lemma-3} &\leq \sum_{r > K \log \log X} \frac{|\theta|^r}{r!} \cdot r \cdot C^r \cdot (\log \log k_X)^r \\
        &\leq \sum_{r > K \log \log X} \exp \left( r \log |\theta| + r \log C + r + \log r - r \log r + r \log \log \log k_X \right).
    \end{align*}
    Because $r > K \log \log X$ inside the summation, we have $$\log r > \log K + \log \log \log X > \log K + \log \log \log k_X.$$ Hence, there exist explicit constants $B_*, B_{**} > 0$ such that for sufficiently large $K$, we have
    \begin{align*}
        \eqref{eq:lemma-3} &\leq \sum_{r > K \log \log X} B_* \exp \left(r \log |\theta| + r \log C + r + \log r - r \log K  \right) \\
        &\leq \sum_{r > K \log \log X} B_* \exp \left( - \frac{1}{2} r \log K \right) \leq B_{**} \exp \left( - \frac{1}{2} \log K \log \log X \right).
    \end{align*}
    Combining all the upper bounds for \eqref{eq:lemma-1}, \eqref{eq:lemma-2}, and \eqref{eq:lemma-3}, and choosing $K > 1$, we have
    \begin{align*}
        0 &\leq \lim_{X \to \infty} \left| \mathbb{E} \left[ \exp \left( \theta \sum_{p \in B(X,C)} g(p) Z_p \right) \right] - \mathbb{E} \left[ \exp \left( \theta \sum_{p \in B(X,C)} g(p) Y_p \right) \right] \right| \\
        &\leq \lim_{X \to \infty} B'' \cdot e^{-\frac{1}{2}(1-b)\log X} + (M+1) \cdot B_{**} \cdot e^{-\frac{1}{2} \log K \log \log X} = 0.
    \end{align*}
\end{proof}
\begin{remark}
    Using Lemma \ref{lemma:MZ9}, one can in fact show that
    \begin{equation*}
        \limsup_{X \to \infty} \frac{1}{\log \log X} \log \left| \mathbb{E} \left[ \exp \left( \theta \sum_{p \in B(X,C)} g(p) Z_p \right) \right] - \mathbb{E} \left[ \exp \left( \theta \sum_{p \in B(X,C)} g(p) Y_p \right) \right] \right| = -\infty.
    \end{equation*}
\end{remark}

\section{Proof}
We are now ready to prove Theorem \ref{thm:Main-MonoidMZ}. Take 
$$B(X, C) = \{p \in \mathcal{P} : |g(p)|<C, \mathbf{N}(p) \leq k_X\}.$$
For any real number $\theta$, 
\begin{align}\label{eq:logEexp}
\begin{split}
    \log \mathbb{E}\left[\exp\left(\theta \sum_{p \in B(X,C)}g(p)Y_p\right) \right] &= \log \prod_{p\in B(X,C)} \mathbb{E}\left[ \exp(\theta g(p)Y_P) \right] \\
    &= \sum_{p \in B(X,C)} \log\left(\frac{1}{\mathbf{N}(p)}e^{\theta g(p)} + 1 - \frac{1}{\mathbf{N}(p)}\right).
\end{split}
\end{align}
For $p$ with sufficiently large $\mathbf{N}(p)$, the Taylor expansion for $\log(1+x)$ shows  
\begin{equation}\label{eq:taylorx+1}
    \log\left(\frac{1}{\mathbf{N}(p)}e^{\theta g(p)} + 1 - \frac{1}{\mathbf{N}(p)}\right) = \frac{1}{\mathbf{N}(p)}e^{\theta g(p)}  - \frac{1}{\mathbf{N}(p)} + O_{C,\theta}\left(\frac{1}{\mathbf{N}(p)^2}\right).
\end{equation}
By Lemma \ref{lemma:Liu-Turan}, 
\begin{equation}\label{eq:sumoverloglog}
    \frac{1}{\log \log X} \sum_{\mathbf{N}(p)\leq X} \frac{1}{\mathbf{N}(p)}=1 \quad \text{as} \quad X \to \infty.
\end{equation}
Further, 
\begin{equation}\label{eq:knratio}
    \lim_{X \to \infty} \frac{\log \log k_X}{\log \log X} =1.
\end{equation}
Using the definition of $\rho_X$, Condition \ref{condition:additive}(2) and equations \eqref{eq:logEexp}, \eqref{eq:taylorx+1}, \eqref{eq:sumoverloglog}, and \eqref{eq:knratio} show 
\begin{equation}\label{eq:Eexpsum}
    \lim_{X \to \infty} \frac{1}{\log \log X} \mathbb{E}\left[\exp (\theta \sum_{p \in B(X,c)} g(p)Y_p)\right] = \int_{-C}^C (e^{\theta y}-1) \rho(dy),
\end{equation}
if $\rho\{C\}=0$ and $\rho\{-C\}$. This latter condition is not an obstacle, since the measure $\rho$ on $\mathbb{R}$ has only countably many atoms. So we can choose infinite, increasing sequences of $C$ for which both $\rho\{C\}=0$ and $\rho\{-C\}=0$. 

Lemma \ref{lemma:MZ9} and \eqref{eq:Eexpsum} imply 
\begin{equation*}
      \lim_{X \to \infty} \frac{1}{\log \log X} \mathbb{E}\left[\exp (\theta \sum_{p \in B(X,c)} g(p)Z_p)\right] = \int_{-C}^C (e^{\theta y}-1) \rho(dy).
\end{equation*}

By the G{\"a}rtner-Ellis Theorem (Theorem \ref{theorem:Gartner-Ellis}),
$$\mathbb{P}\left( \frac{\sum_{p \in B(X,c)} g(p)Z_p}{\log \log X} \in A\right)$$
satisfies a large deviation principle with rate function 
$$I_C(x) = \sup_{\theta \in \mathbb{R}} \left\{ \theta x - \int_{-C}^C (e^{\theta y}-1)\rho(dy)\right\}.$$
Taking $C \to \infty$, and using Lemmas \ref{lemma:MZ7} and \ref{lemma:MZ8} to handle the tail error estimates on $I_C(X)$, we have that 
$$\mathbb{P}\left( \frac{W(X)}{\log \log X} \in A\right)$$
satisfies a large deviation principle with rate function 
$$I(x) = \lim_{C \to \infty}I_C(x) = \sup_{\theta \in \mathbb{R}} \left\{ \theta x - \int_{-\infty}^\infty (e^{\theta y}-1)\rho(dy)\right\},$$
as desired.

\section{Acknowledgments}
This work was completed while the first author visited the second at the Max Planck Institute for Mathematics in Bonn; we thank MPIM for its generous support and for providing excellent working conditions.  We also thank Sachin Kumar for pointing out a typo in the main theorem. 

\nocite{*}
\bibliographystyle{alpha}
\bibliography{arxivbib.bib}
\end{document}